\documentclass[flen,a4paper]{article}
\usepackage{amsmath,amsthm,amssymb,amscd}
\usepackage{graphicx}
\usepackage{graphics}
\usepackage{verbatim}
\usepackage{enumerate}
\usepackage{mathrsfs}
\usepackage{color}
\usepackage[top=1in, bottom=0.8in, left=0.8in, right=0.8in]{geometry}

\makeatletter

\newcommand{\Rmnum}[1]{\expandafter\@slowromancap\romannumeral #1@}

\newtheorem{theorem}{Theorem}[section]
\newtheorem{proposition}[theorem]{Proposition}

\newtheorem{lemma}[theorem]{Lemma}
\newtheorem{definition}[theorem]{Definition}

\newtheorem{conjecture}[theorem]{Conjecture}

\newtheorem{question}[theorem]{Question}
\makeatother

\begin{document}
	
	\title{A study on parity signed graphs: the $rna$ number}
	
	\vspace{3cm}
	\author{Ligang Jin\footnotemark[1]~, Xiaoyue Chen\footnotemark[1]~, Yingli Kang\footnotemark[2]}
	\footnotetext[1]{Department of Mathematics,
		Zhejiang Normal University, Yingbin Road 688,
		321004 Jinhua,
		China;
		ligang.jin@zjnu.cn (L. Jin), 1113985495@qq.com (X. Chen); Supported by ZJNSF LY20A010014 and NSFC U20A2068.}
	\footnotetext[2]{Corresponding author. Department of Mathematics, Jinhua Polytechnic, Western Haitang Road 888, 321017 Jinhua, China; ylk8mandy@126.com; Supported by NSFC 11901258.}
	\date{}
	
	\maketitle

\begin{abstract}	
The study on parity signed graphs was initiated by Acharya and Kureethara very recently and then followed by Zaslavsky etc..
Let $(G,\sigma)$ be a signed graph on $n$ vertices. If
$(G,\sigma)$ is switch-equivalent to $(G,+)$ at a set of $\lfloor \frac{n}{2} \rfloor$ many vertices, then we call $(G,\sigma)$ a parity signed graph and $\sigma$ a parity-signature.
$\Sigma^{-}(G)$ is defined as the set of the number of negative edges of $(G,\sigma)$ over all possible parity-signatures $\sigma$.
The $rna$ number $\sigma^-(G)$ of $G$ is given by $\sigma^-(G)=\min \Sigma^{-}(G)$.
In other words,  $\sigma^-(G)$ is the smallest cut size that has nearly equal sides. 

In this paper, all graphs considered are finite, simple and connected. We apply switch method to the characterization of parity signed graphs and the study on the $rna$ number.
We prove that: for any graph $G$, $\Sigma^{-}(G)=\left\{\sigma^{-}(G)\right\}$ if and only if $G$ is $K_{1, n-1} $ with $n$ even or $K_{n}$. This confirms a conjecture proposed 
in [M. Acharya and J.V. Kureethara. Parity labeling in signed graphs. J. Prime Res. Math., to appear. arXiv:2012.07737].
Moreover, we prove a nontrivial upper bound for the $rna$ number: for any graph $G$ on $m$ edges and $n$ ($n\geq 4$) vertices, $\sigma^{-}(G)\leq \lfloor \frac{m}{2}+\frac{n}{4} \rfloor$. We show that $K_n$, $K_n-e$ and $K_n-\triangle$ are the only three graphs reaching this bound. This is the first upper bound for the $rna$ number so far. Finally, we prove that: for any graph $G$, $\sigma^-(G)+\sigma^-(\overline{G})\leq \sigma^-(G\cup \overline{G})$, where $\overline{G}$ is the complement of $G$. This solves a problem proposed in 
[M. Acharya, J.V. Kureethara and T. Zaslavsky. Characterizations of some parity signed graphs. 2020, arXiv:2006.03584v3].
\end{abstract}


\section{Introduction}\label{sec_intro}
All graphs considered in this paper are finite, simple and connected.
Denote by $P_{n}$, $I_{n}$ and $K_n$ the path, the independent set and the complete graph on $n$ vertices, respectively. Denote by $K_n-e$, $K_n-2e$, $K_n-P_2$ and $K_n-\triangle$ the graph obtained from $K_n$ by removing an edge, removing two nonadjacent edges, removing two adjacent edges, and removing three edges which induce a triangle, respectively.
For any two vertex-disjoint graphs $G$ and $K$, denote by $G\vee K$ the graph obtained from $G$ and $K$ by connecting each vertex of $G$ to each vertex of $H$. 
Let $V_1$ and $V_2$ be two disjoint set of vertices of a graph $G$. Denote by $E_G(V_1,V_2)$ the set of edges connecting a vertex of $V_1$ to a vertex of $V_2$.
When $G$ is clear from the context, we use $E(V_1,V_2)$ for short.

\subsection{Parity signed graphs}
Signed graph is a kind of extension of graph, and the research on signed graph is one of the hot issues in graph theory in recent years, such as flow of signed graphs (eg. \cite{CQZhang_2021_flow,Raspaud_Zhu_2011}), homomorphism of signed graphs (eg. \cite{Reza_2014_homomorphism,Reza_2021_homomorphism}), coloring of signed graphs (eg. \cite{Jin_2016_choosability,Kang_steffen_2018_circular,Raspaud_2016_chromatic_number,Zaslavsky_1982}), circuit cover of signed graphs (eg. \cite{CQZhang_2019_circuit_cover}) and so on. 
The concept of signed graph was first proposed by Harary \cite{s4} in 1953. 
A \emph{signed graph} $(G,\sigma)$ is a graph $G$ together with a mapping $\sigma: E(G)\rightarrow \{1,-1\}$. We call $G$ the \emph{underlying graph} and $\sigma$ a \emph{signature} of $G$.
An edge $e$ is \emph{positive} if $\sigma(e)=1$, and \emph{negative} otherwise.
Denote by $(G,+)$ the signed graph $(G,\sigma)$ with $\sigma(e)=1$ for each $e\in E(G)$. 
Let $S=(G,\sigma)$ be a signed graph. The set of all negative edges of $S$ is denoted by $E^{-}(S)$.
Let $v\in V(S)$. By $d^{-}(v)$ (resp., $d^{+}(v)$) we mean the number of negative (resp., positive) edges of $S$ incident with $v$. Define that $d^\Delta(v)=d^-(v)-d^+(v)$ and call $d^\Delta(v)$ the \emph{sign-difference} of $v$.

Parity signed graph is a particular class of signed graphs. The concept of parity signed graphs was introduced very recently by Acharya and Kureethara \cite{s2}.
A signed graph $S=(G, \sigma)$ is called a \emph{parity signed graph} if there exists a bijection $f: V(G)\rightarrow\left\{1, 2, \cdots, n\right\}$ such that for each edge $e=uv$ of $G$, $f(u)$ and $f(v)$ have the same parity if $\sigma(e)=1$, and opposite parities if $\sigma(e)=-1$. Such an $f$ (resp., $\sigma$) is called a \emph{parity-labeling} (resp., \emph{parity-signature}) of $S$, or of $G$.
In the paper \cite{s6}, Acharya, Kureethara and Zaslavsky characterized parity signed  graphs as follows: A signed graph $S$ is a parity signed graph if and only if its vertex set $V(S)$ can be bipartitioned into two sets $V_{1}(S)$ and $V_{2}(S)$ such  that $||V_{1}(S)|-|V_{2}(S)||\leq1$ and that any two adjacent vertices $u$ and $v$ of $S$ belong to the same set if and only if the edge $uv$ is positive.
In this paper, such a bipartition is called a \emph{parity-partition}, and $V_{1}(S)$ and $V_{2}(S)$ are correspondingly called \emph{parity sets}.
The characterization above can be easily observed by taking $V_1(S)=\{v\in V(G)\colon f(v)\text{~is odd}\}$ and $V_2(S)=\{v\in V(G)\colon f(v) \text{~is even}\}$.
In what follows, we will use the language of parity-partition very often for parity signed graphs.

\subsection{The $rna$ number $\sigma^-(G)$}
In the paper \cite{s2}, the authors introduced the concept of the $rna$ number of a parity signed graph. However, according to its definition, it is more reasonable to say the $rna$ number of a graph. 
Let $G$ be a graph. Denote by $\Sigma^{-}(G)$ the set of the number of negative edges of a parity signed graph $(G,\sigma)$ over all possible parity-signatures $\sigma$.
The \emph{$rna$ number} $\sigma^{-}(G)$ of $G$ is defined as the minimum number of negative edges of a parity signed graph $(G,\sigma)$ over all possible parity-signatures $\sigma$, i.e., $\sigma^{-}(G)=\min \Sigma^{-}(G)$.

In the paper \cite{s2}, Acharya and Kureethara put forward the following conjecture on the $rna$ number and verified it for the class of trees.
\begin{conjecture}[\cite{s2}, Conjecture 1]\label{conj-AK}
	Let $G$ be a parity signed graph. Then $\sigma^-(G)=|E^-(G)|$ if and only if $G$ is either $K_{1,n}$, $n$ odd or $K_n$.	
\end{conjecture} 
This conjecture is formulated a little bit weird since under the assumption that ``Let $G$ be a parity signed graph'', we can not say ``$G$ is either $K_{1,n}$, $n$ odd or $K_n$'', for which the parity signature of $G$ has not been given. 
We can check that the argument for the verification of this conjecture for trees (\cite{s2}, Theorem 10) is actually a proof of the following conjecture for trees, which is formulated slightly different from Conjecture \ref{conj-AK}. 
\begin{conjecture}\label{conj}
	Let $G$ be a graph on $n$ vertices. Then $\Sigma^{-}(G)=\left\{\sigma^{-}(G)\right\}$ if and only if $G$ is $K_{1, n-1} $ with $n$ even or $K_{n}$.	
\end{conjecture} 
We will confirm the truth of Conjecture \ref{conj} in Section \ref{sec_conjecture}.
There are two crucial ideas for the proof: the application of a graph operation called ``parity-switch'', and the characterization of parity signed graphs that are ``degree-balanced''.
We will introduce and study the concept of parity-switch in Section \ref{sec_parity_switch}, and we will introduce the concept of degree-balance and address a complete list of parity signed graphs that are degree-balanced in Section \ref{sec_degree_balanced}.

Moreover, we study upper bounds for the $rna$ number in Section \ref{sec_bound}.
A result of the paper \cite{s6} (Theorem 3.4) states that for any natural number $k$, there is a graph $G$ with $\sigma^-(G)=k$. Hence, there exist no constant upper bounds for $\sigma^-$ in general.
It is known that $\sigma^-(K_n)=\lceil\frac{n}{2}  \rceil \lfloor \frac{n}{2} \rfloor$ (\cite{s2}, Proposition 9).
Actually, $\lceil\frac{n}{2}  \rceil \lfloor \frac{n}{2} \rfloor$ is a trivial upper bound for the $rna$ number of graphs on $n$ vertices.
In Section \ref{sec_bound}, we will prove a much better upper bound than the trivial one. More precisely, we prove that for any graph $G$ on $m$ edges and $n$ ($n\geq 4$) vertices, $\sigma^{-}(G)\leq \lfloor \frac{m}{2}+\frac{n}{4} \rfloor$. Additionally, 
we completely characterize parity signed graphs which approach the bound $\lceil\frac{n}{2}  \rceil \lfloor \frac{n}{2} \rfloor$ and which achieve the bound $\lfloor \frac{m}{2}+\frac{n}{4} \rfloor$. The proof is done by taking a parity signature of $G$ with minimum number of negative edges and applying an argument on the sign-difference of vertices for this parity signed graph.

In section \ref{sec_complement}, we solve a problem on the parity complement proposed in \cite{s6}.
\section{Our results}\label{sec_result}

\subsection{Parity-switches} \label{sec_parity_switch}
Switch is an important tool for the study of signed graphs.  Let $S=(G, \sigma)$ be a signed graph and $v\in V(G)$. Denote by $E(v)$ the set of edges incident with $v$. \emph{Switching} at $v$ gives a new signed graph $S^{\prime}=(G, \sigma^{\prime})$ defined as $\sigma^{\prime}(e)=-\sigma(e)$ if $e\in E(v)$; and $\sigma^{\prime}(e)=\sigma(e)$ otherwise. Two signed graphs are called \emph{switch-equivalent} if one can be obtained from the other by a sequence of switches.

For a parity signed graph, a switch at a pair of vertices with opposite parity labels (equivalently, a pair of vertices locating in different parity sets) is called a $\emph{parity-switch}$.	In this paper, parity-switch plays a crucial role in the proof of our results.

The following proposition, which follows directly from definition, gives a characterization of parity signed graphs from the perspective of switch.  
Hence, we can easily list all the parity signed graphs for any given underlyling graph.
\begin{proposition}\label{pro_all_positive}
	A signed graph $(G,\sigma)$ on $n$ vertices is a parity signed graph if and only if it can be obtained from $(G,+)$ by switching at a set of vertices of cardinality $\lfloor \frac{n}{2} \rfloor$.	
\end{proposition}

From any given parity signed graph, the following proposition provides a way to generate new parity signed graphs whose underlying graph is same as before.

\begin{proposition}\label{pro_a_switch}
	Given a parity signed graph $S=(G,\sigma)$, the graph $S^{\prime}=(G, \sigma^{\prime})$ obtained from $S$ by a parity-switch is also a parity signed graph.	
\end{proposition}

\begin{proof}
	Exchange the labels of the two vertices we switch at.
	It is easy to verify by definition that the resulting labeling is a parity-labeling of $S'$.
\end{proof}

Combining the previous two propositions, one can easily deduce the following proposition, which shows that parity-switch is the exact way to describe the transformation between any two parity signed graphs having the same underlying graph.
\begin{proposition}\label{pro_switches}
	Let $(G,\sigma)$ be a parity signed graph.	Then a signed graph $(G,\sigma')$ is a parity signed graph if and only if it can be obtained from $(G,\sigma)$ by a sequence of parity-switches.
\end{proposition}
\begin{proof}
	The sufficiency follows by applying Proposition \ref{pro_a_switch} multiple times.
	For the necessity, by Proposition \ref{pro_all_positive}, both $(G,\sigma)$ and $(G,\sigma')$ can be obtained from $(G,+)$ by switching at some set of vertices of cardinality $\lfloor \frac{|V(G)|}{2} \rfloor$, say $V_1$ and $V_2$, respectively.
	So, on one hand, $(G,\sigma')$ can be obtained from $(G,\sigma)$ by switching at $(V_1-V_1\cap V_2)\cup (V_2-V_1\cap V_2)$. On the other hand, the two sets $V_1-V_1\cap V_2$ and $V_2-V_1\cap V_2$ have the same cardinality and different parities under the parity labeling of $(G,\sigma)$. Therefore, we can conclude that the switch from $(G,\sigma)$ to $(G,\sigma')$ above consists of a sequence of parity-switches, which completes the proof.
\end{proof}

\subsection{Degree-balanced parity signed graphs}\label{sec_degree_balanced}
\begin{definition}
	Let $S$ be a parity signed graph with parity sets $V_1$ and $V_2$. $S$ is \emph{degree-balanced} if for any two vertices $u\in V_1$ and $v\in V_2$,
	\begin{equation*}
	\begin{cases}
	d^\triangle (u)+d^\triangle(v)= 2, \text{~if~} uv\in E(S);\\
	d^\triangle (u)+d^\triangle(v)= 0, \text{~if~} uv\notin E(S).	
	\end{cases}
	\end{equation*}
\end{definition}

\begin{lemma}\label{lem_degree_balanced}
	Let $S=(G,\sigma)$ be a parity signed graph on $n$ vertices. If $S$ is degree-balanced, then $G$ must be $K_{1,n-1}$ or $K_n$ when $n$ is even; and $G$ must be $P_2\vee I_{n-2}$ or $K_n$ when $n$ is odd.
\end{lemma}

\begin{proof}
	Let $V_1$ and $V_2$ be the two parity sets of $S$. Since $S$ is degree-balanced, by definition, for any two vertices $u\in V_1$ and $v\in V_2$,
	\begin{equation} \label{eq_d_equal}
	\begin{cases}
	d^\triangle (u)+d^\triangle(v)= 2, \text{~if~} uv\in E(S);\\
	d^\triangle (u)+d^\triangle(v)= 0, \text{~if~} uv\notin E(S).	
	\end{cases}
	\end{equation}
	
	For $i\in\{1,2\}$, let $V_{i1}$ consist of the vertices of $V_i$ which has a neighbor not in $V_i$, and let $V_{i2}=V_i\setminus V_{i1}$.
	Since $G$ is connected, $V_{11}\neq \emptyset$ and $V_{21}\neq \emptyset$.	
	However, both $V_{12}$ and $V_{22}$ might be empty sets.
	So, we distinguish two cases as follows.
	
	Case 1:  Assume $V_{12}=V_{22}=\emptyset$. In this case, $V_i=V_{i1}$ for $i\in\{1,2\}$.
	We will show that $E(V_1,V_2)$ is complete. Suppose to the contrary that there exist $x\in V_1$ and $y\in V_2$ such that $xy\notin E(G)$.
	By the definition of $V_{i1}$, $x$ has a neighbor (say $x'$) in $V_2$ and $y$ has a neighbor (say $y'$) in $V_1$. Applying Formula (\ref{eq_d_equal}) to pairs $\{x,y\}$, $\{x,x'\}$ and $\{y,y'\}$ gives that
	\begin{equation*}
	\begin{cases}
	d^\triangle(x)+d^\triangle(y)=0,\\
	d^\triangle(x)+d^\triangle(x')=2,\\
	d^\triangle(y)+d^\triangle(y')=2.
	\end{cases}
	\end{equation*}
	It follows that $d^\triangle(x')+d^\triangle(y')=4,$ a contradiction to Formula (\ref{eq_d_equal}).
	This complete the proof that $E(V_1,V_2)$ is complete.
	
	For any vertices $x\in V_1$ and $y\in V_2$, Formula (\ref{eq_d_equal}) implies that $d^\triangle(x)+d^\triangle(y)=2$. Notice that $d^-(x)+d^-(y)=n$.
	So,
	$$d^+(x)+d^+(y)=d^-(x)+d^-(y)-d^\triangle(x)-d^\triangle(y)=n-2.$$
	This means that $x$ is adjacent to all the other vertices in $V_1$, and so does $y$ in $V_2$. By the arbitrariness of $x$ and $y$, both $V_1$ and $V_2$ induce complete graphs in $G$. Recall that $E(V_1,V_2)$ is complete.
	So, $G$ is the complete graph $K_n$, as desired.

	Case 2:\quad  Assume $V_{12}\neq{\emptyset}$ or $V_{22}\neq{\emptyset}$. W.l.o.g., let us say $V_{22}\neq{\emptyset}$.
	
    Take an arbitrary $z\in V_{22}$ and let $d^\Delta(z)=-t$.
	For any $x\in V_1$, we have $xz\notin E(S)$ by the definition of $V_{22}$. It follows from Formula (\ref{eq_d_equal}) that $d^\Delta(x)=t$.
	For any $y\in V_{21}$, since $y$ has a neighbor (say $y'$) in $V_1$, applying Formula (\ref{eq_d_equal}) to $y$ and $y'$ gives that $d^\Delta(y)=-t+2$. Now, $d^\Delta(x)+d^\Delta(y)=2$. By applying Formula (\ref{eq_d_equal}) to $x$ and $y$, we have $xy\in E(S)$.
	By the arbitrariness of $x$ and $y$, $E(V_1,V_{21})$ is complete.
	
	Clearly, $d^-(y)=|V_1|$ and $d^-(z)=0$.
	Since $y$ has at most $|V_2|-1$ neighbors in $V_2$, we have $d^+(y)\leq |V_{2}|-1$. So,
	\begin{equation} \label{eq_y}
	-t+2=d^\Delta(y)=d^-(y)-d^+(y)\geq|V_{1}|-|V_{2}|+1.
	\end{equation}
	Since $z$ has at least one neighbor in $V_2$, $d^+(z)\geq 1$.
	So,
	\begin{equation}\label{eq_t_geq_1}
	-t=d^\Delta(z)=d^-(z)-d^+(z)\leq -1.
	\end{equation}
	Combining Equations ($\ref{eq_y}$) and ($\ref{eq_t_geq_1}$) gives $|V_{1}|-|V_{2}|\leq 0$.
	By definition, $|V_{1}|-|V_{2}|\in\{-1,0,1\}$. So, it suffices to distinguish the following two cases.
	
	Case 1.1: Assume $|V_{1}|-|V_{2}|=0$. In this case, $n$ is even.
	It follows from Equations (\ref{eq_y}) and (\ref{eq_t_geq_1}) that $t=1$.
	So, we can compute that $d^+(y)=\frac{n}{2}-1$ and $d^+(z)=1$.
	In particular, the former equality implies that $yz\in E(S)$.
	Then it follows from the latter equality that such $y$ is unique, i.e., $|V_{21}|=1$.
	Notice that $y$ is the only neighbor of $z$. So, $V_{22}$ is an independent set of $G$.
	Moreover, $d^+(x)=d^-(x)-d^\triangle(x)=1-1=0$. It follows that $V_{1}$ is an independent set of $G$.
	Notice that $y$ is adjacent to all the vertices of $V_{1}\cup V_{22}$.
	So, we can conclude that $G$ is $K_{1,n-1}$, as desired.
	
	Case 1.2: Assume $|V_{1}|-|V_{2}|=-1$. In this case, $n$ is odd. Let $p=|V_{21}|$ and $q=|V_{22}|$. Clearly, $p+q=\frac{n+1}{2}.$
	It follows from Equations (\ref{eq_y}) and (\ref{eq_t_geq_1}) that $t\in\{1,2\}$.
	We distinguish two cases by the value of $t$.
	
	Let $t=1$. So, we can compute that $d^{+}(y)=\frac{n-3}{2}$ and $d^{+}(z)=1$.
	On one hand, the sum
	$$\sum_{v\in V_2}d^+(v)=p\times \frac{n-3}{2} + q\times 1=p(p+q-2)+q$$
	must be even by the Hand-shaking Theorem. It follows that both $p$ and $q$ are even.
	On the other hand, since $z$ has at most one neighbor in $V_{21}$, $|E(V_{21},V_{22})|\leq q$. Moreover, since $y$ has at least $q-1$ neighbors in $V_{22}$, $|E(V_{21},V_{22})|\geq p(q-1)$.
    These two inequalities lead to
	\begin{equation*}
	q\geq p(q-1).
	\end{equation*}
	Recall that both $p$ and $q$ are even. We can compute from this inequality that
	$p=q=2$. So, $|V_1|=3$ and for each $x\in V_1$, we have $d^{+}(x)=d^{-}(x)-d^\triangle(x)=p-t=1$.
	So, the sum $\sum_{v\in V_1}d^+(v)=3$ is an odd integer, a contradiction to the Hand-shaking Theorem.
	
	Let $t=2$. Similarly, we can deduce that $d^{+}(y)=\frac{n-1}{2}$ and  $d^{+}(z)=2$. The former equality implies that $y$ is adjacent to all other vertices of $V_2$. Then it follows from the latter equality that $|V_{21}|\leq 2$.
	Moreover, $|V_{21}|=d^-(x)\geq d^-(x)-d^+(x)=d^\triangle(x)=t=2$.
	So, $|V_{21}|=2$.
	We can compute that $d^{+}(x)=0$.
	It follows from $d^{+}(z)=2$ and $d^{+}(x)=0$ that $V_1$ and $V_{22}$ are independent sets of $G$, respectively.
	Now, we can see that $G$ is $ P_{2} \vee I_{n-2}$, as desired.
\end{proof}
	
\subsection{The proof of Conjecture \ref{conj}}\label{sec_conjecture}
\begin{proof}
    Sufficiency has already been proved  (see \cite{s2}, Proposition 9 and Theorem 10).
		
	Necessity. Take any parity signature $\sigma$ of $G$ and let $S=(G,\sigma)$. Let $V_1$ and $V_2$ be the two parity sets of $S$.
	By Proposition \ref{pro_a_switch}, switching at a vertex of $V_1(S)$ and a vertex of $V_2(S)$ (i.e., a parity-switch) in $S$ results into another parity signed graph. As $\Sigma^{-}(G)=\left\{\sigma^{-}(G)\right\}$, this switch will not change the number of negative edges of the whole graph. Thus, $S$ is degree-balanced.
	By Lemma \ref{lem_degree_balanced}, $G$ must be $K_{1,n-1}$ or $K_n$ when $n$ is even; and $G$ must be $P_2\vee I_{n-2}$ or $K_n$ when $n$ is odd.
	Moreover, it is easy to compute that $\Sigma^-(P_2\vee I_{n-2})=\{n-1,n+1\}$.
	Hence, $G$ can not be $P_2\vee I_{n-2}$ since the assumption that $\Sigma^{-}(G)=\left\{\sigma^{-}(G)\right\}$.
	This completes the proof of the conjecture.
\end{proof}

\subsection{Upper bounds for the $rna$ number}\label{sec_bound}
The following theorem gives a trivial upper bound for the $rna$ number as well as a complete characterization of graphs whose $rna$ number is very close to this bound.
\begin{theorem}\label{thm-trivial-bound}
	Let $G$ be a graph on $n$ vertices. Then $\sigma^-(G)\leq \lceil\frac{n}{2}  \rceil \lfloor \frac{n}{2} \rfloor$ and the following statements hold:
	\begin{enumerate}
	\setlength{\itemsep}{0pt}
	\item $\sigma^-(G)=\lceil\frac{n}{2}  \rceil \lfloor \frac{n}{2} \rfloor$ if and only if $G=K_n$;
	\item $\sigma^-(G)=\lceil\frac{n}{2}  \rceil \lfloor \frac{n}{2} \rfloor - 1$ if and only if $G=K_n-e$;
	\item $\sigma^-(G)=\lceil\frac{n}{2}  \rceil \lfloor \frac{n}{2} \rfloor - 2$ if and only if $G\in \{K_n-\triangle, K_n-2e, K_n-P_2\}$.
    \end{enumerate}
\end{theorem}

\begin{proof}
	Let $S=(G,\sigma)$ be a parity signed graph with minimum $|E^-(S)|$, and let $V_1$ and $V_2$ be the two parity sets of $S$.
	Then $\sigma^-(G)\leq |E^-(S)| = |E(V_1,V_2)| \leq \lceil\frac{n}{2}  \rceil \lfloor \frac{n}{2} \rfloor$.
	
	For the three statements, let us first prove the sufficiency.
	For $G=K_n$, there are precisely $\lceil\frac{n}{2}  \rceil \lfloor \frac{n}{2} \rfloor$ edge between $V_1$ and $V_2$ Hence, $\sigma^-(K_n)=\lceil\frac{n}{2}  \rceil \lfloor \frac{n}{2} \rfloor$.
	For $G=K_n-e$, we have $|E^-(S)|=\lceil\frac{n}{2}  \rceil \lfloor \frac{n}{2} \rfloor$ if $e\in E(V_1,V_2)$; and $|E^-(S)|=\lceil\frac{n}{2}  \rceil \lfloor \frac{n}{2} \rfloor -1$ otherwise. 
	Hence, $\sigma^-(K_n-e)=\lceil\frac{n}{2}  \rceil \lfloor \frac{n}{2} \rfloor - 1$. 
	Similarly, we can prove that $\sigma^-(K_n-2e)=\sigma^-(K_n-P_2)=\lceil\frac{n}{2}  \rceil \lfloor \frac{n}{2} \rfloor - 2$.
	For $G=K_n-\triangle$, we have $|E(\triangle)\cap E(V_1,V_2)|\in\{0,2\}$, which implies that $|E^-(S)|\in\{\lceil\frac{n}{2} \rceil \lfloor \frac{n}{2} \rfloor - 2, \lceil\frac{n}{2} \rceil \lfloor \frac{n}{2} \rfloor\}$.
	Hence, $\sigma^-(K_n-\triangle)= \lceil\frac{n}{2} \rceil \lfloor \frac{n}{2} \rfloor - 2$.
	
    It remains to prove the necessity. For the first statement, suppose to the contrary that $G\neq K_n$. Let $xy\neq E(G)$. Take a choice of $V_1$ and $V_2$ such that $x\in V_1$ and $y\in V_2$. So, $\sigma^-(G)\leq |E^-(S)|\leq \lceil\frac{n}{2}  \rceil \lfloor \frac{n}{2} \rfloor - 1$, a contradiction.
    For the second statement, suppose to the contrary that $G\neq K_n-e$. Since the truth of the first statement, $G\neq K_n$. Hence, we may let $e_1$ and $e_2$ be two edges (may have a common end) not in $G$. Take a choice of $V_1$ and $V_2$ such that $e_1,e_2\in E(V_1,V_2)$. So, $\sigma^-(G)\leq |E^-(S)|\leq \lceil\frac{n}{2}  \rceil \lfloor \frac{n}{2} \rfloor - 2$, a contradiction.
    For the third statement, suppose to the contrary that $G\neq \{K_n-2e, K_n-P_2, K_n-\triangle\}$. Since the statements 1 and 2,  we may let $e_1, e_2$ and $e_3$ be three edges not in $G$ which do not induce a triangle. Then we can take a choice of $V_1$ and $V_2$ such that $e_1,\ldots,e_3\in E(V_1,V_2)$. So, $\sigma^-(G)\leq |E^-(S)|\leq \lceil\frac{n}{2}  \rceil \lfloor \frac{n}{2} \rfloor - 3$, a contradiction.
\end{proof}

 For any graph on $m$ edges and n vertices, since $m\leq \frac{n(n-1)}{2}$, we have $\lfloor \frac{m}{2}+\frac{n}{4} \rfloor \leq \lceil\frac{n}{2}  \rceil \lfloor \frac{n}{2} \rfloor$.
 The following theorem shows that $\lfloor \frac{m}{2}+\frac{n}{4} \rfloor$ is also an upper bound for the $rna$ number. Moreover, the graphs which achieve this bound are all determined. 
\begin{theorem}\label{thm_bound}
	For any graph $G$ on $m$ edges and $n$ ($n\geq 4$) vertices, $\sigma^{-}(G)\leq \lfloor \frac{m}{2}+\frac{n}{4} \rfloor$. Moreover, the equality holds if and only if $G$ is $K_n$ or $K_n-e$ or $K_n-\triangle$.
\end{theorem}

\begin{proof}
Take a parity signature $\sigma$ of $G$ with minimum number of negative edges, i.e., $|E^-(S)|=\sigma^-(G)$, where $S=(G,\sigma)$. Let $V_1$ and $V_2$ be the two parity sets of $S$.
For $i\in\{1,2\}$, let $X_i=\{uv\colon u,v\in V_i \text{~and~} uv\notin E(G)\}$.
Let $Y=\{uv\colon u\in V_1, v\in V_2 \text{~and~} uv\notin E(G)\}$.
By Proposition \ref{pro_a_switch}, switching at a vertex of $V_1$ and a vertex of $V_2$ (i.e., a parity-switch) in $S$ results into another parity signed graph. As $|E^{-}(S)|=\sigma^{-}(G)$, this switch will not decrease the number of negative edges of the whole graph. Thus, for any $u\in V_1$ and $v\in V_2$,
\begin{equation} \label{eq_d}
\begin{cases}
d^\triangle (u)+d^\triangle(v)\leq 2, \text{~if~} uv\in E(S);\\
d^\triangle (u)+d^\triangle(v)\leq 0, \text{~if~} uv\notin E(S).	
\end{cases}
\end{equation}
We distinguish two cases by the parity of $n$.

 Case 1: Assume that $n$ is even.

 Denote by $H$ the complete bipartite graph whose partite sets are $V_1$ and $V_2$.
 Take a perfect matching $M$ of $H$ with minimum $|M\setminus Y|$. Let $M'=M\setminus Y$ and $M''=M\cap Y$.
 It follows from Formula (\ref{eq_d}) that
\begin{equation}\label{eq_M}
 \begin{split}
 \sum_{v\in V(G)}d^\triangle(v)
&=\sum_{xy\in M}(d^\triangle(x)+d^\triangle(y))\\
&=\sum_{xy\in M'}(d^\triangle(x)+d^\triangle(y))
+\sum_{xy\in M''}(d^\triangle(x)+d^\triangle(y))\\
&\leq \sum_{xy\in M'}2
+\sum_{xy\in M''}0\\
&=2|M'|,
 \end{split}
\end{equation}
Moreover, the left side of Formula (\ref{eq_M}) can be written as
 \begin{equation}\label{eq_MLeft}
 \begin{split}
 \sum_{v\in V(G)}d^\triangle(v)
 &=\sum_{v\in V(G)}d^-(v)-\sum_{v\in V(G)}d^+(v)\\
 &=2|E^-(S)|-2|E^+(S)|\\
 &=4|E^-(S)|-2m,
 \end{split}
 \end{equation}
 where last equality uses the fact
 $|E^-(S)|+|E^+(S)|=m.$
 Combining Formulas (\ref{eq_M}) and (\ref{eq_MLeft}) and the fact $|M'|\leq \frac{n}{2}$ yields
 \begin{equation}\label{eq_bound}
 \sigma^-(G)=|E^-(S)|\leq \frac{m}{2}+\frac{|M'|}{2}\leq \frac{m}{2}+\frac{n}{4}.
 \end{equation}
 Since $\sigma^-(G)$ is an integer, it follows that
 \begin{equation}\label{eq_bound_integer}
 \sigma^-(G)\leq \lfloor \frac{m}{2}+\frac{n}{4}\rfloor,
 \end{equation}
 as desired.

Now we consider the equality case of Formula (\ref{eq_bound_integer}).
The sufficiency is easy to verified by using Theorem (\ref{thm-trivial-bound}).
Next, we will prove the necessity. Assume that $\sigma^-(G)= \lfloor \frac{m}{2}+\frac{n}{4}\rfloor$.

 Notice that $\frac{m}{2}+\frac{n}{4}$ may not be an integer. So, we distinguish two cases as follows.

 Case 1.1: First assume that $\frac{m}{2}+\frac{n}{4}$ is an integer, i.e.,
 $\sigma^-(G)= \frac{m}{2}+\frac{n}{4}$.
 It follows from Formula (\ref{eq_bound}) that $|M'|=\frac{n}{2}$.
 By the minimality of $|M'|$, $E(V_1,V_2)$ is complete, i.e., $E^-(S)=|E(V_1,V_2)|=\frac{n^2}{4}.$ Combining this with the assumption $E^-(S)=\sigma^-(G)=\frac{m}{2}+\frac{n}{4}$ yields $m=\frac{n(n-1)}{2}$, i.e., $G$ is the complete graph $K_n$, as desired.

 Case 1.2: It remains to assume that $\frac{m}{2}+\frac{n}{4}$ is not an integer. Since $n$ is even,
 $\sigma^-(G)= \frac{m}{2}+\frac{n}{4}-\frac{1}{2}$.
 It follows from Formula (\ref{eq_bound}) that $|M'|\in\{\frac{n}{2}-1, \frac{n}{2}\}$.

 Case 1.2.1: Let $|M'|=\frac{n}{2}$. By the minimality of $|M'|$, $E(V_1,V_2)$ is complete, i.e., $E^-(S)=|E(V_1,V_2)|=\frac{n^2}{4}.$ Combining this with the assumption $E^-(S)=\sigma^-(G)=\frac{m}{2}+\frac{n}{4}-\frac{1}{2}$ yields $m=\frac{n(n-1)}{2}-1$, i.e., $G$ is the graph $K_n-e$, as desired.

 Case 1.2.2: Let $|M'|=\frac{n}{2}-1$.
 By the minimality of $|M'|$, we can deduce that $Y$ induces a star.
 Denote by $w$ the center of $G[Y]$. W.l.o.g., let $w\in V_1$. Denote that $N(w)=\{y\colon\ wy\in Y\}$.
 If $|Y|=1$, then $\sigma^-(G)=(\frac{n}{2})^2-1.$ Combining this with the assumption that $\sigma^-(G)= \frac{m}{2}+\frac{n}{4}-\frac{1}{2}$, we have $m=\frac{n(n-1)}{2}-1$, i.e., $G$ is $K_n-e$, as desired.
 Hence, we may assume that $|Y|\geq 2$.
 It follows that, for any $x\in V_1$ and $y\in V_2$ such that $(x,y)\notin \{(w,v)\colon\ v\in V_2\setminus N(w)\}$, there always exists a choice for $M$ such that $xy\in M$.
 By putting the values $|M'|=\frac{n}{2}-1$ and $\sigma^-(G)= \frac{m}{2}+\frac{n}{4}-\frac{1}{2}$ into Formulas (\ref{eq_MLeft}) and (\ref{eq_M}) in turn,
 we will see that the equality of Formula (\ref{eq_M}) holds. In particular,
  \begin{equation} \label{eq_almost_db}
 d^\triangle (x)+d^\triangle(y)=
 \begin{cases}
  2, \text{~if~} xy\in E(S);\\
  0, \text{~if~} xy\notin E(S).	
 \end{cases}
 \end{equation}
 If $N(w)=V_2$, then by the arbitrariness of $x$ and $y$, it follows that $G$ is degree-balanced.
 Notice that $n$ is even.
 By Lemma \ref{lem_degree_balanced}, $G$ must be $K_{1,n-1}$ or $K_n$.
 Since $Y$ induces a star, $G$ can not be $K_{1,n-1}$ with $n\neq 4$ or $K_n$.
 Therefore, $G$ must be $K_{1,3}$, i.e., $K_4-\triangle$, as desired.
 Hence, we may next assume that $N(w)\subsetneq V_2$. Since again $Y$ induces a star, by the definition of $d^\triangle$, we have $d^\triangle(v)\geq 0$ for $v\in N(w)$; and $d^\triangle(v)\geq 1$ for $v\in V\setminus(N(w)\cup \{w\})$.
 Combining this with Formula (\ref{eq_almost_db}), we can deduce that $d^\triangle(w)= -1$ and $d^\triangle(v)=1$ for $v\in V(G)\setminus \{w\}$.
 It follows from the latter equality that $d^+(v)=\frac{n}{2}-1$ for $v\in V\setminus(N(w)\cup \{w\})$ and $d^+(v)=\frac{n}{2}-2$ for $v\in V(w)$. 
 That is to say, $X_1=\emptyset$ and $X_2$ induces a matching whose ends form the set $N(w)$. So, $d^+(w)=\frac{n}{2}-1$. Recall that $d^\triangle(w)= -1$, we thereby obtain $\frac{n}{2}-|Y|=d^-(w)=d^\triangle(w)+d^+(w)=\frac{n}{2}-2$. So, $|Y|=2$. This implies that $X_1\cup X_2\cup Y$ induce a triangle. So, $G$ is $K_n-\triangle$, as desired.

 Case 2: Assume that $n$ is odd. W.l.o.g., let $|V_1|<|V_2|$, i.e., $|V_1|=\frac{n-1}{2}$ and $|V_2|=\frac{n+1}{2}$.

Notice that $\sum_{x\in V_1}d^-(x)$ counts each edge of $E(V_1,V_2)$ precisely once.
So,
$$\sum_{x\in V_1}d^-(x)=\frac{(n-1)(n+1)}{4}-|Y|.$$
Similarly, since $\sum_{x\in V_1}d^+(x)$ counts each edge of $G[V_1]$ precisely twice,
we have
$$\sum_{x\in V_1}d^+(x)=2(\frac{(n-1)(n-3)}{8}-|X_1|).$$
Hence, by the definition of the function $d^\triangle$, we can compute that
\begin{equation}\label{value_1}
\begin{split}
\sum_{x\in V_1}d^\triangle(x)
&=\sum_{x\in V_1}d^-(x)-\sum_{x\in V_1}d^+(x)\\
&=(\frac{(n-1)(n+1)}{4}-|Y|)-2(\frac{(n-1)(n-3)}{8}-|X_1|)\\
&=n-1+2|X_1|-|Y|.
\end{split}
\end{equation}
Similarly, we can deduce that
\begin{equation}\label{value_2}
\begin{split}
\sum_{y\in V_2}d^\triangle(y)
&=(\frac{(n-1)(n+1)}{4}-|Y|)-2(\frac{(n+1)(n-1)}{8}-|X_2|)\\
&=2|X_2|-|Y|.
\end{split}
\end{equation}
Now, the following sum, denoted by $D$, can be written as
\begin{equation}\label{value_D1}
\begin{split}
D
&=\sum_{\substack{x\in V_1\\ y\in V_2}}(d^\triangle(x)+d^\triangle(y))\\
&=\frac{n+1}{2}\sum_{x\in V_1}d^\triangle(x)+\frac{n-1}{2}\sum_{y\in V_2}d^\triangle(y)\\
&=\frac{(n-1)(n+1)}{2}-n|Y|+(n+1)|X_1|+(n-1)|X_2|,
\end{split}
\end{equation}
where the last equality uses Formulas (\ref{value_1}) and (\ref{value_2}).

By the definition of $X_1, X_2$ and $Y$, we have
\begin{eqnarray}
|X_1|+|Y|+|X_2|+m=\frac{n(n-1)}{2} \label{identity_1},\\
|Y|+\sigma^-(G)=\frac{(n-1)(n+1)}{4} \label{identity_2}.
\end{eqnarray}
From these two equations, we can easily deduce that
\begin{equation}\label{identity_3}
|X_1|+|X_2|-|Y|=2\sigma^-(G)-\frac{n-1}{2}-m.
\end{equation}
Now, Formula (\ref{value_D1}) can be further written as follows:
\begin{equation}\label{value_D2}
\begin{split}
D
&=2(|Y|+\sigma^-(G))-n|Y|+(n+1)|X_1|+(n-1)|X_2|\\
&=2\sigma^-(G)+(n-2)(|X_1|+|X_2|-|Y|)+3|X_1|+|X_2|\\
&=2\sigma^-(G)+(n-2)(2\sigma^-(G)-\frac{n-1}{2}-m)+3|X_1|+|X_2|.
\end{split}
\end{equation}
where the first equality uses Formula (\ref{identity_2}) and the last equality uses Formula (\ref{identity_3}).

On the other hand, it follows from Formula (\ref{eq_d}) that \begin{equation}\label{value_D3}
\begin{split}
D
&=\sum_{\substack{x\in V_1\\ y\in V_2}}(d^\triangle(x)+d^\triangle(y))\\
&\leq 2\sigma^-(G),
\end{split}
\end{equation}
where the equality holds if and only if $G$ is degree-balanced.
Combining Formulas (\ref{value_D2}) and (\ref{value_D3}) yields that
\begin{equation}\label{value_sigma}
\sigma^-(G)\leq \frac{n-1}{4}+\frac{m}{2}.
\end{equation}
Since $\sigma^-(G)$ is an integer and $n$ is odd, we further have
\begin{equation}\label{value_sigma_floor}
\sigma^-(G)\leq \lfloor \frac{n}{4}+\frac{m}{2} \rfloor,
\end{equation}
as desired.

 Now we consider the equality case of Formula (\ref{value_sigma_floor}).
 The sufficiency is easy to verified by using Theorem \ref{thm-trivial-bound}.
 Next, we will prove the necessity. Assume that $\sigma^-(G)=\lfloor \frac{n-1}{4}+\frac{m}{2} \rfloor$.
Notice that $\frac{n-1}{4}+\frac{m}{2}$ may be an integer. So, we distinguish two cases as follows.

Case 2.1: Assume that $\frac{n-1}{4}+\frac{m}{2}$ is an integer, i.e.,
$\sigma^-(G)= \frac{n-1}{4}+\frac{m}{2}$. Put this value into Formula (\ref{value_D2}),
we thereby have $D=2\sigma^-(G)+3|X_1|+|X_2|$. Combining this with Formula (\ref{value_D3}) gives $|X_1|=|X_2|=0$. So, the equality of Formula (\ref{value_D3}) holds, that is, $G$ is degree-balanced. By Lemma \ref{lem_degree_balanced}, $G$ must be $P_2\vee I_{n-2}$ or $K_n$. Since $|X_1|=|X_2|=0$, we can easily check that $G$ can not be $P_2\vee I_{n-2}$ with $n\geq 5$. Notice that $P_2 \vee I_1$ is $K_3$. Therefore, we can conclude that $G$ must be $K_n$, as desired.

Case 2.2: Assume that $\frac{n-1}{4}+\frac{m}{2}$ is not an integer, i.e.,
$\sigma^-(G)= \frac{n-1}{4}+\frac{m}{2}-\frac{1}{2}$. In this case, we can deduce from Formulas (\ref{identity_1}) and (\ref{identity_2}) that
\begin{equation}\label{relation}
|Y|=|X_1|+|X_2|+1.
\end{equation}

We shall show that $Y$ induces a star. Let $z$ be any vertex of $V_2$. If $d^\triangle(z)\geq 1$, then moving $z$ into $V_1$ results into another parity-partition of $G$, whose negative edges are less than $\sigma^-(G)$, a contradiction. Hence,
\begin{equation}\label{dz}
d^\triangle(z)\leq 0.
\end{equation}
Denote by $H(z)$ the complete bipartite graph whose partite sets are $V_1$ and $V_2\setminus \{z\}$.  Denote by $t(z)$ the maximum value of $|M\cap Y|$ over all the perfect matchings $M$ of $H(z)$.
\begin{equation}\label{sum1}
\begin{split}
\sum_{v\in V(G)}d^\triangle(v)
&=\sum_{xy\in M}(d^\triangle(x)+d^\triangle(y))+d^\triangle(z)\\
&\leq t(z)\times 0 + (\frac{n-1}{2}-t(z))\times 2 +d^\triangle(z)\\
&=n-1-2t(z)+d^\triangle(z).
\end{split}
\end{equation}
On the other hand, adding up Formulas (\ref{value_1}) and (\ref{value_2}) gives
\begin{equation}\label{sum2}
\begin{split}
\sum_{v\in V(G)}d^\triangle(v)
&=\sum_{x\in V_1}d^\triangle(x)+\sum_{y\in V_2}d^\triangle(y)\\
&=n-1+2(|X_1|+|X_2|-|Y|)\\
&=n-3,
\end{split}
\end{equation}
where the last equality uses Formula (\ref{relation}).
Combining Formulas (\ref{dz}), (\ref{sum1}) and (\ref{sum2}) gives
\begin{equation}\label{tz}
t(z)\leq 1.
\end{equation}
Notice that $|Y|\geq 1$, implied by Formula (\ref{relation}).
By the arbitrariness of $z$ and the definition of $t$, it follows from Formula (\ref{tz}) that $Y$ must induce a star.

If $|Y|=1$, then Formula (\ref{relation}) implies that $|X_1|=|X_2|=0$, i.e., $G$ is the graph $K_n-e$, as desired.

Hence, we may next assume that $|Y|\geq 2$.
Denote by $w$ the center of the star $G[Y]$.
It follows that for any $y\in V_2\setminus\{w\}$, we have $t(y)=1$ by the definition of $t$. Combining this  with Formulas (\ref{sum1}) and (\ref{sum2}) gives $d^\triangle(y)\geq 0.$ By Formula (\ref{dz}), we further have
\begin{equation}\label{dy=0}
d^\triangle(y)=0.
\end{equation}
We distinguish two cases.

Case 2.2.1: Assume that $w\in V_1$. It follows from Formula (\ref{dy=0}) that $\sum_{y\in V_2}d^\triangle(y)=0$, i.e., $|Y|=2|X_2|$ by Formula (\ref{value_2}).
Moreover, for any $x\in V_1\setminus \{w\}$, applying Formula (\ref{eq_d}) to $x$ and $y$ gives $d^\triangle(x)\leq 2$.
Recall that $Y$ induces a star. So, $d^-(x)=\frac{n+1}{2}$.  We can conclude that $d^+(x)=d^-(x)-d^\triangle(x)\geq \frac{n-3}{2}$, i.e., $x$ is adjacent to all other vertices of $V_1$. By the arbitrariness of $x$, it follows that $|X_1|=0$.
Put the values $|Y|=2|X_2|$ and $|X_1|=0$ into Formula (\ref{relation}), we thereby compute that $|X_2|=1$. Now, we can see that $G$ is $K_n-\triangle$, as desired.

Case 2.2.2: Assume that $w\in V_2$. Recall that $Y$ induces a star. In this case, for any $y\in V_2\setminus\{w\}$, $d^-(y)=\frac{n-1}{2}$. Combining this with Formula (\ref{dy=0}) yields $d^+(y)=d^-(y)-d^\triangle(y)=\frac{n-1}{2}$, i.e., $y$ is adjacent to all other vertices of $V_2$. By the arbitrariness of $y$, it follows that $|X_2|=0$.
Moreover, for any $x\in V_1$, applying Formula (\ref{eq_d}) to $x$ and $y$ gives $d^\triangle(x)\leq 2$. So, $\sum_{x\in V_1}d^\triangle(x)\leq 2\times|V_1|=n-1$.
Combining this with Formula (\ref{value_1}) gives $2|X_1|\leq |Y|.$
By putting Formula (\ref{relation}) and $|X_2|=0$ into this inequality, we have $|X_1|\leq 1$.
For $|X_1|=0$, we can see that $G$ is $K_n-e$, as desired.
For $|X_1|=1$, $G$ is $K_n-\triangle$, as desired.
\end{proof}

Remark: For graphs of order at most 4 (i.e., $n\leq 4$), it is easy to check that only $K_2$ and $K_3$ do not satisfy the upper bound given by Theorem \ref{thm_bound}.

\subsection{Parity complement}\label{sec_complement}
Denote by $\overline{G}$ the complement of a graph $G$. The \emph{parity complement} $\overline{S}_p$ of a parity signed graph $S$ is the parity signed graph whose underlying graph is the complement of the one of $S$ and whose parity sets are same as the one of $S$. This definition is formulated slightly different from but equivalent to the one in \cite{s6}. 

In this section, we answer the following question on the $rna$ number of the parity complement of a parity signed graph, proposed in \cite{s6}.
\begin{question}
For any parity signed graph $S$, what is the relation between $\sigma^-(S)+\sigma^-(\overline{S}_p)$ and $\sigma^-(S\cup \overline{S}_p)$?	
\end{question}
Since the $rna$ number of a parity signed graph is actually defined for its underlying graph, this question can be equivalently formulated as: for any graph $G$, what is the relation between $\sigma^-(G)+\sigma^-(\overline{G})$ and $\sigma^-(G\cup \overline{G})$?
We answer this question by the following theorem.
\begin{theorem}
For any graph $G$, $\sigma^-(G)+\sigma^-(\overline{G})\leq \sigma^-(G\cup \overline{G})$. The equality holds if and only if $G$ is $K_{1, n-1} $ with $n$ even or $K_{n}$.		
\end{theorem}
\begin{proof}
Let $n=|V(G)|$.
Clearly, $\sigma^-(G\cup \overline{G})=\sigma^-(K_n)=\lfloor \frac{n}{2}\rfloor \lceil \frac{n}{2}\rceil$.
Let $S$ be any parity signed graph whose underlying graph is $G$.
By the definition of parity complement, $|E^-(S)|+|E^-(\overline{S}_p)|=\lfloor \frac{n}{2}\rfloor \lceil \frac{n}{2}\rceil.$ Since the arbitrariness, it follows that $\max \Sigma^-(G) + \min \Sigma^-(\overline{G})=\lfloor \frac{n}{2}\rfloor \lceil \frac{n}{2}\rceil.$ Therefore, $\sigma^-(G)+\sigma^-(\overline{G})\leq \max \Sigma^-(G) + \min \Sigma^-(\overline{G})= \sigma^-(G\cup \overline{G})$.
Clearly, the equality holds if and only if $\sigma^-(G)=\max \Sigma^-(G)$, i.e., $\Sigma^-(G)=\{\sigma^-(G)\}$.
By the truth of Conjecture \ref{conj}, this is equivalent to that $G$ is $K_{1, n-1} $ with $n$ even or $K_{n}$.
\end{proof}

\end{document}